\documentclass[12pt]{amsart}
\usepackage[margin=1.3in]{geometry}

\usepackage{hyperref}

\usepackage{amsrefs}

\usepackage[foot]{amsaddr}

\usepackage{pgfplots}
\pgfplotsset{compat = newest}

\usepackage{todonotes}
\usepackage[T1]{fontenc}

\usepackage{amsmath}
\usepackage{amsthm}
\usepackage{amssymb}

\usepackage{mathtools}

\DeclareMathOperator{\sech}{sech}

\theoremstyle{plain}
\newtheorem{theorem}{Theorem}

\newtheorem{lemma}{Lemma}

\newtheorem{proposition}{Proposition}

\theoremstyle{definition}
\newtheorem{definition}{Definition}
\newtheorem{example}{Example}
\newtheorem{remark}{Remark}
\newtheorem{problem}{Problem}

\usepackage{graphicx, color, caption}
\usepackage{xcolor}
\graphicspath{{fig/}}

\usepackage{mathrsfs} 

\usepackage{wrapfig}
\usepackage{float}

\title{An Application of Optimal Control Theory to R-Tipping}
\author{Grace Z. Zhang}
\address{University of Minnesota}
\email{zhan5640@umn.edu}
\date{\today}

\begin{document}
	
	\begin{abstract}
	An application of optimal control theory results in a lower bound on the speed $|\dot{\lambda}(t)|$ that must be attained at least once by any external forcing function that induces tipping in the asymptotically autonomous scalar ODE $\dot{x} = f(x+\lambda(t))$. The value of this critical speed depends on the total arclength $\int_{-\infty}^{\infty} |\dot{\lambda}(t)| ~dt$ of forcing, and may be interpreted as a safe threshold rate associated to each given arclength, such that if the speed of forcing remains everywhere slower than this, tipping cannot occur. The bound is tight in the sense that there exists a forcing function (continuous but non-smooth) which induces tipping, possesses the required arclength, and never exceeds the threshold speed. Further, the threshold speed is a strictly decreasing function of arclength, thus capturing the trade off between \textit{how fast} and \textit{how far} of a minimal disturbance characterizes tipping. 
	\end{abstract}

\maketitle

The standard setting for rate-induced \cites{ashwinTippingPointsOpen2012, ashwinParameterShiftsNonautonomous2017, wieczorekRateinducedTippingThresholds2023} involves fixing a particular parameterized family of smooth forcing functions and identifying a critical value of the rate parameter. In contrast, we consider a broad collection of all possible forcing functions, continuous but not necessarily smooth, and seek a general property possessed by those which effect tipping behavior. We focus on rigidly shifting asymptotically autonomous scalar systems $\dot{x}= f (x +\lambda(t))$ and identify a nonsmooth choice of forcing function $\lambda(t)$ which is an optimal tipping strategy in the sense that it utilizes the least possible maximum speed. Under a co-moving change of coordinates, the problem of finding this optimal $\lambda(t)$ becomes dual to the problem of finding an additive control function that achieves basin escape with minimum fuel. We show the optimizer is a bang-bang control. 

We separate the presentation into a scalar special case where the forcing function is assumed to be monotone and the basin of attraction one-sided, followed by the general scalar case with these assumptions removed. 

\section{Introductory Examples}

\subsection{Smooth Prototype}

\begin{example}\label{ex:prototype}A prototypical example of rate-induced tipping, given in \cite{ashwinParameterShiftsNonautonomous2017}, involves a base vector field $\dot{x} = x^2-1$ which is nonautonomously shifted to the left by a smooth ramp function $\lambda(rt)$. A fixed constant $\lambda_{\infty}>2$ defines the total amplitude of the shift, while a variable rate parameter $r>0$ modulates its steepness, with smaller $r$ corresponding to a slower shift and larger $r$ to a faster shift. 
	\[
	\begin{gathered}
		\dot{x} = (x + \lambda)^2 -1\\
		\lambda(rt) = \frac{\lambda_{\infty}}{2}\left(1+\tanh\left(\frac{\lambda_{\infty}rt}{2}\right)\right)
	\end{gathered}
	\]
	
\begin{figure}[H]
	\begin{center}
		\begin{tikzpicture}[]
			\begin{axis}[  axis lines=center,
				xtick={1, -1,-2, -4},
				hide obscured x ticks=false,
				major tick length=15,
				xticklabels={1, -1, 1-$\lambda_{\infty}$, -1-$\lambda_{\infty}$},
				ticklabel style={font=\small, yshift=-1.5ex, color=black, rotate=-40},
				ytick=\empty,
				every tick/.style={black, thick},
				xmax=3,
				xmin=-6,
				ymax=3,
				ymin=-2,
				yscale=0.5,
				xlabel=$x$,
				]
				\addplot[smooth, ultra thick, blue] {x^2-1};
				\addplot[smooth, dotted, thick, blue] {(x+3)^2-1};
			\end{axis}
			
			\draw[->, blue]        (3,2.5)   -- (2,2.5);
			\draw[->, blue]        (3.3,1.7)   -- (2.3,1.7);
			\draw[->, blue]        (3.8,0.9)   -- (2.8,0.9);
		\end{tikzpicture}
	\end{center}
	\vspace{0.1 in}
	\begin{center}
		\begin{tikzpicture}[]
			\begin{axis}[  axis lines=center,
				ticks=none,
				xmax=5,
				xmin=-5,
				ymax=3.5,
				ymin=-0.5,
				yscale=0.5,
				xlabel=$t$,
				]
				\addplot[smooth, ultra thick, red] {3/2*(1+tanh(3*1*x/2))};
				\addplot[dotted, thick, black] {3};
			\end{axis}
			\node[] at (7.5,2.5) {$\lambda_{\infty}$};
			\node[] at (-0.3,0.4) {0};
		\end{tikzpicture}
	\end{center}
		\caption{A prototypical example of rate-induced tipping.}
	\end{figure}
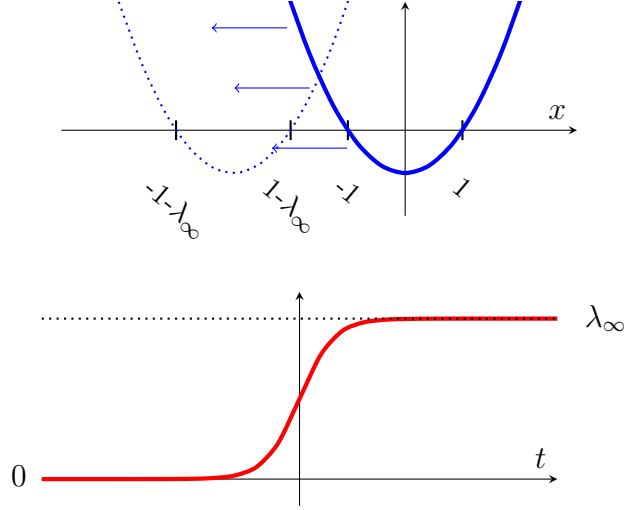
	
	As the vector field translates rigidly leftwards, the attracting equilibrium at $x=-1$ is displaced. Rate-induced tipping concerns itself with whether a trajectory beginning at the original steady state can adapt to this displacement or if it becomes destabilized. 
	
	Here it is known that for each parameter regime $(\lambda_{\infty}, r)$ there exists a unique solution $\hat{x}(t)$ to the ODE such that $\lim\limits_{t\to - \infty}\hat{x}(t) = -1$. Further, fixing $\lambda_{\infty}$, there exists a critical value $r=r_{c}(\lambda_{\infty})$ such that
	$$\begin{cases}
		\lim\limits_{t\to \infty}\hat{x}(t) = -1-\lambda_{\infty} & \text{for } r < r_c\\
		\lim\limits_{t\to \infty}\hat{x}(t) = 1-\lambda_{\infty}& \text{for } r=r_c\\
		\hat{x}(t) \to \infty \text{ (in finite time) } & \text{for } r > r_c
	\end{cases},
	$$
		and an exact expression for this critical value  \cite{perrymanHowFastToo2015} is known to be
	$$r_{c} = \frac{4}{\lambda_{\infty}(\lambda_{\infty}-2)}.$$
	
	Intuitively, a sufficiently slow shift allows the trajectory to seamlessly "track" the moving attractor. But a too-fast shift destabilizes it onto the other side of the moving repeller. There is a critical rate in between where the trajectory ends up balanced precisely on the basin boundary. 

\end{example}

\subsection{Piecewise Linear Prototype}

Most literature on rate-induced tipping customarily assumes smoothness of the ramping function; however, the next instance of a nonsmooth, piecewise linear ramp will be of core importance to this work.

\begin{example}
	Replace the smooth ramping function in Example \ref{ex:prototype} with the following piecewise linear ramping function, where the slope of the increasing portion is $m>0$, and as before $\lambda_{\infty} >2$. 
	
	\[
	\begin{gathered}
		\dot{x} = (x + \lambda)^2 -1\\
		\lambda(mt) = \begin{cases}
			0 & \text{if } t < 0\\
			mt & \text{if } 0 \leq t \leq \lambda_{\infty}/m\\
			0 & \text{if } t > \lambda_{\infty}/m
		\end{cases}
	\end{gathered}
	\]

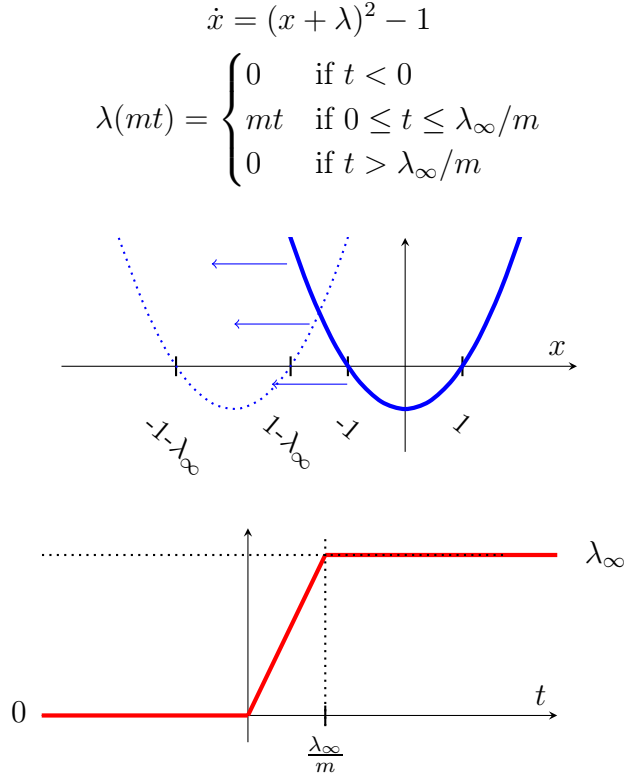
\begin{figure}[H]
		\begin{center}
		\begin{tikzpicture}[]
			\begin{axis}[  axis lines=center,
				xtick={1, -1,-2, -4},
				hide obscured x ticks=false,
				major tick length=15,
				xticklabels={1, -1, 1-$\lambda_{\infty}$, -1-$\lambda_{\infty}$},
				ticklabel style={font=\small, yshift=-1.5ex, color=black, rotate=-40},
				ytick=\empty,
				every tick/.style={black, thick},
				xmax=3,
				xmin=-6,
				ymax=3,
				ymin=-2,
				yscale=0.5,
				xlabel=$x$,
				]
				\addplot[smooth, ultra thick, blue] {x^2-1};
				\addplot[smooth, dotted, thick, blue] {(x+3)^2-1};
			\end{axis}
			
			\draw[->, blue]        (3,2.5)   -- (2,2.5);
			\draw[->, blue]        (3.3,1.7)   -- (2.3,1.7);
			\draw[->, blue]        (3.8,0.9)   -- (2.8,0.9);
		\end{tikzpicture}
	\end{center}
	\vspace{0.1 in}
	\begin{center}
		\begin{tikzpicture}[]
			\begin{axis}[  axis lines=center,
				xtick={3/2},
				hide obscured x ticks=false,
				major tick length=15,
				xticklabels={$\frac{\lambda_{\infty}}{m}$},
				ticklabel style={font=\small, color=black},
				ytick=\empty,
				every tick/.style={black, thick},
				xmax=6,
				xmin=-4,
				ymax=3.5,
				ymin=-0.5,
				yscale=0.5,
				xlabel=$t$,
				]
				\addplot[smooth, ultra thick, red, domain=-4:0] {0};
				\addplot[smooth, ultra thick, red, domain=0:3/2] {2*x};
				\addplot[smooth, ultra thick, red, domain=3/2:6] {3};
				\addplot[dotted, thick, black] {3};
				\addplot[dotted, thick] coordinates {(3/2, -0.2) (3/2, 3.3)};
			\end{axis}
			\node[] at (7.5,2.5) {$\lambda_{\infty}$};
			\node[] at (-0.3,0.4) {0};
		\end{tikzpicture}
	\end{center}
			\caption{A piecewise linear ramping function.}
	\end{figure}
	
	The formal setting for this and related non-smooth ODEs will be reviewed later; however, it can be shown that tipping behavior analogous to the previous example occurs.  For each parameter regime $(\lambda_{\infty}, m)$ the ODE possesses a unique solution $\hat{x}(t)$ such that $\lim\limits_{t\to - \infty}\hat{x}(t) = -1$.\footnote{In this case,  $\hat{x}(t) = -1$ for all $t \in (-\infty, 0]$} Fixing $\lambda_{\infty}$, there exists a critical value $m=m_c(\lambda_{\infty})$ with 
	
	\[\begin{cases}
		\lim\limits_{t\to \infty}\hat{x}(t) = -1-\lambda_{\infty} & \text{for } m < m_c\\
		\lim\limits_{t\to \infty}\hat{x}(t) = 1-\lambda_{\infty} \footnote{} & \text{for } m=m_c\\
		\hat{x}(t) \to \infty \text{ (in finite time) } & \text{for } m > m_c
	\end{cases}.
	\]
	\footnotetext{In this case $\hat{x}(t) = 1-\lambda_{\infty}$ for all $t \in [\lambda_{\infty}/m, \infty)$.}
	For this example, it can be shown that $m_c$ is given by the unique solution to the equation
	\[\frac{2m_c}{\sqrt{m_c-1}}\arctan\left(\frac{1}{\sqrt{m_c-1}}\right)=\lambda_{\infty}.\]
	We remark that $m_c$ is a strictly decreasing function of $\lambda_{\infty}$; that is, a larger amplitude shift allows for a more gentle critical rate of shift while a smaller amplitude shift requires a steeper critical rate of shift. A similar amplitude-rate trade off between $r_c$ and $\lambda_{\infty}$ may be observed in Example \ref{ex:prototype}, for instance by focusing on the maximum slope of the critical ramp function, which occurs at $t=0$.
	
	\label{ex:piecewise}
\end{example}

Each of the two previous examples involves fixing a particular family of ramp functions, parameterized by a variable $r$ or $m$ that controls the overall steepness of the ramp, and then identifying a critical ramp from within the predetermined family. This is the standard point of view that prevails in the rate-induced tipping literature. Instead, we would now like to consider the entire collection of possible functions that interpolate asymptotically from 0 to $\lambda_{\infty}$, subject to appropriate conditions, and seek a general property about the steepness of those which effect tipping behavior.

Our core insight is that the critical piecewise linear ramp function of Example \ref{ex:piecewise} is actually an optimal tipping strategy in an important sense. We will show that any arbitrary scalar ramp function $\lambda(t)$ (monotone non-decreasing, for now) that induces tipping when applied to the same base vector field must attain a slope greater than or equal to $m_c(\lambda_{\infty})$ at least once. The result is a necessary but not sufficient criterion for tipping, or equivalently a safe threshold for non-tipping. 

\section{Change to Co-Moving Coordinates}

For  $\dot{x} = f(x+\lambda(t))$, we consider a change of coordinates to co-moving coordinates,
\begin{align*}
		y&=x+\lambda\\
		\implies \dot{y} &= f(y) + \dot{\lambda}(t)
\end{align*}

This transformation has the effect of converting the ramp function that originally translated the base vector field $\dot{x} = f(x)$ leftwards into a pulse function that now translates the same base vector field up and then back down.

\begin{figure}[H]
\begin{tikzpicture}[]
	\begin{axis}[  axis lines=center,
		ticks=none,
		xmax=3,
		xmin=-6,
		ymax=3,
		ymin=-2,
		yscale=0.5,
		xlabel=$x$,
		]
		\addplot[smooth, ultra thick] {x^2-1};
		\addplot[smooth, dotted, thick] {(x+3)^2-1};
	\end{axis}
	
	\draw[->]        (3,2.5)   -- (2,2.5);
	\draw[->]        (3.3,1.7)   -- (2.3,1.7);
	\draw[->]        (3.8,0.9)   -- (2.8,0.9);
	\node[] at (3.5,-0.5) {(a)};
\end{tikzpicture}
\hspace{0.5in}\begin{tikzpicture}[]
	\begin{axis}[  axis lines=center,
		ticks=none,
		xmax=3,
		xmin=-3,
		ymax=6,
		ymin=-2,
		xscale=2/3 ,
		yscale=4/5,
		xlabel=$x$,
		]
		\addplot[smooth, ultra thick, domain=-2:2] {x^2-1};
		\addplot[smooth, dotted, thick, domain=-2:2] {(x)^2-1+3};
	\end{axis}
	
	\draw[->]        (1.5,1.6)   -- (1.5,2.3);
	\draw[->]        (3.1,2.3)   -- (3.1,1.6);
	
	\node[] at (2.3,-0.5) {(b)};
\end{tikzpicture}
	\caption{A change to co-moving coordinates converts the leftward translating ramp function to an up-and-down translating pulse function.}
\end{figure}
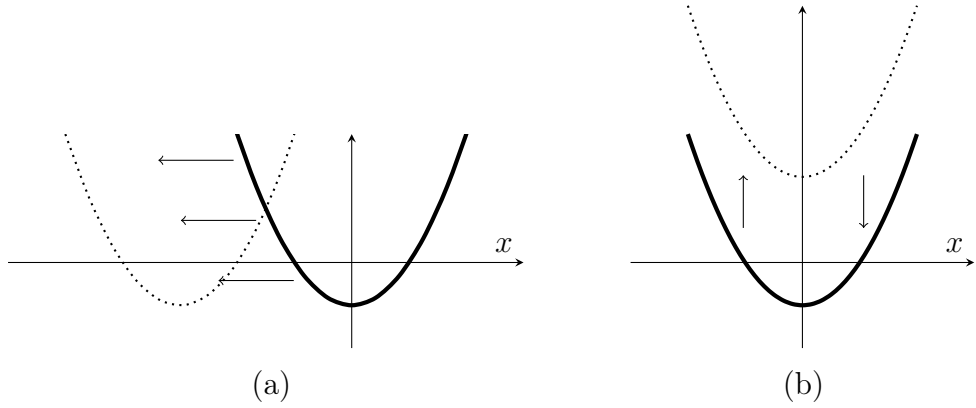

\subsection{Smooth Prototype}

\begin{example}
	For $\lambda(rt)$ as in Example \ref{ex:prototype}  the pulse function is
	$$\dot{\lambda}(rt) = \left( \frac{\lambda_{\infty}}{2}\right)^2r \sech ^2 \left(\frac{\lambda_{\infty}rt}{2}\right)$$
		where, fixing $\lambda_{\infty}$, a smaller value of the parameter $r$ corresponds to a shorter and wider peak, while a larger value of the parameter $r$ corresponds to a taller and narrower peak. 
	
\begin{figure}[H]
	\begin{center}
		
		\begin{tikzpicture}[]
			\begin{axis}[  axis lines=center,
				ticks=none,
				xmax=5,
				xmin=-5,
				ymax=3,
				ymin=-0.5,
				yscale=0.5,
				xlabel=$t$,
				]
				\addplot[smooth, ultra thick, red] {9/4 * (1/cosh (3*x/2))^2};
			\end{axis}
		\end{tikzpicture}
		
	\end{center}
\caption{The pulse function that results from the smooth prototype ramp.}
	\end{figure}
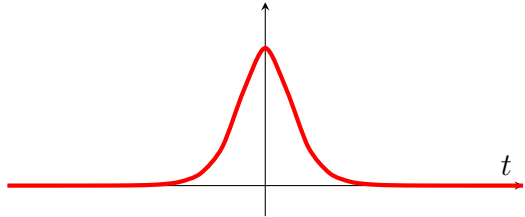
	
Regardless of the choice of $r$, note that the total area under the pulse is unaffected, since it is always equal to $\lambda_{\infty}$. 
\end{example}

\subsection{Piecewise Linear Prototype}

\begin{example}
	\label{ex:step}
	For $\lambda(mt)$ as in Example \ref{ex:piecewise} we obtain a discontinuous step pulse
	$$\dot{\lambda}(mt) = \begin{cases}
		0 & \text{if } t < 0\\
		m & \text{if } 0 \leq t \leq \lambda_{\infty}/m\\
		0 & \text{if } t > \lambda_{\infty}/m
	\end{cases}.$$

\begin{figure}[H]
	\begin{center}
		\begin{tikzpicture}[]
			\begin{axis}[  axis lines=center,
				xtick={3/2},
				ytick = {2},
				major tick length=15,
				xticklabels={$\frac{\lambda_{\infty}}{m}$},
				yticklabels={$m$},
				ticklabel style={font=\small, color=black},
				every tick/.style={black, thick},
				xmax=5,
				xmin=-3,
				ymax=3.3,
				ymin=-0.5,
				yscale=0.5,
				xlabel=$t$,
				]
				\addplot[smooth, ultra thick, red, domain=-3:0] {0};
				\addplot[smooth, ultra thick, red, domain=0:3/2] {2};
				\addplot[smooth, ultra thick, red, domain=3/2:5] {0};
				
				\draw[color=black, dotted, thick] 
				(axis cs:3/2, 0) -- (axis cs:3/2, 2);
			\end{axis}
		\end{tikzpicture}
	\end{center}
	\caption{The pulse function that results from the piecewise linear ramp.}
	\end{figure}
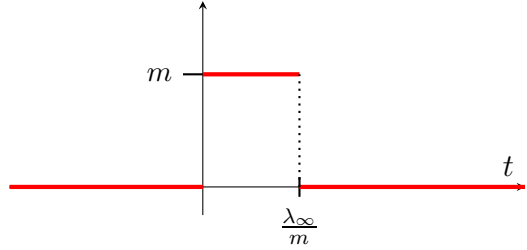
	
	The formal setting for this non-smooth change of coordinates between a non-smooth vector field and a discontinuous vector field will be discussed soon. But here we can see that smaller $m$ corresponds to a shorter and wider step while larger $m$ results in a taller and narrower step, again while always preserving the area $\lambda_{\infty}$ underneath. Here, the critical step has a height of $m_c = m_c(\lambda_{\infty})$ as defined in Example \ref{ex:piecewise}. 
\end{example}

In the co-moving frame of reference, our desired assertion becomes the statement that any arbitrary tipping pulse whose total area equals $\lambda_{\infty}$ must at some point reach or surpass the height $m_c(\lambda_{\infty})$ of the critical step from Example \ref{ex:step}.

We adopt the point of view that an "arbitrary pulse" is a measurable and essentially bounded (and non-negative, for now) control function $u(t)$ added to the base vector field to obtain the nonautonomous ODE
\begin{equation*} 
	\dot{y} =f(y)+u(t) 
\end{equation*} 
Then, fixing the restriction $\int_{-\infty}^{\infty}u(t) ~dt = \lambda_{\infty}$, we would wish to demonstrate that the essential supremum of any control $u$ that induces tipping is at least the height of the critical step pulse from Example \ref{ex:step}. Actually, in order to cast this problem into an optimization problem with a more amenable cost function and constraint, we will instead prove a near-converse before subsequently recovering the full result. 

\begin{remark}
	A control function which simply toggles between a minimum and maximum value, such as the critical step function we have described, is commonly known as a \textbf{bang-bang control}. Bang-bang control arises as an optimal control in several contexts \cite{artsteinDiscreteContinuousBangBang1980}. 
\end{remark}

\section{Formal Setting}

We set our attention on control systems of the form
\begin{equation}
	\dot{y} = f(y)+u(t).
	\label{eq:control}
\end{equation}

Assume that $f: \mathbb{R} \to \mathbb{R}$ is $C^2$. Assume that the control function $u: \mathbb{R} \to \mathbb{R}$ is in the space  $L^\infty$ of measurable and essentially bounded functions. In particular, the measurability of $u$ implies that it is allowed some points of discontinuity, but is continuous almost everywhere, while essential boundedness means bounded almost everywhere. Here the norm is 
$$||u||_\infty = \inf\{C \geq 0  :  |u(t)| \leq C  \text{ for almost every } t \in \mathbb{R} \}.$$ 

It remains to be justified whether this construction produces well-defined solutions. The right hand side of the ODE satisfies the Carath\'eodory conditions for existence and uniqueness of solutions on $\mathbb{R} \times \mathbb{R}$ \cite{haleOrdinaryDifferentialEquations1980}:

\begin{itemize}
	\item For every fixed $t$, $f(y) + u(t)$ is clearly continuous in $y$.
	\item For every fixed $y$, $f(y) + u(t)$ is clearly measurable in $t$.
	\item $u(t)$ is essentially bounded on $\mathbb{R}$ and $f(y)$ is continuous hence bounded on every compact set $K\subset \mathbb{R}$,  so $|f(y) +u(t)| \leq \sup_K |f| + \text{ ess}~\sup_\mathbb{R} |u|$ almost everywhere. The bound is a constant function thus Lebesgue-integrable on $K \times \mathbb{R}$.
\end{itemize} 

This is enough to guarantee local existence of an absolutely continuous solution $y(t)$ to any initial value problem $y(t_0) = y_0$ in the extended sense that
\[\begin{gathered}
	y(t) = y(t_0) + \int\limits_{t_0}^t f(y(s)) + u(s) ~ds\\
	\text{and } \dot{y}(t) = f(y) + u(t) \text{ almost everywhere.}
\end{gathered}\]

Furthermore, $f$ is $C^2$ hence locally Lipschitz continuous, so fixing any $t$ gives a locally Lipschitz continuous $f(y) + u(t)$,  where the local Lipschitz constant is clearly not affected by the choice of $t$, hence is uniform in $t$. This guarantees uniqueness of the local solution. 

\subsection{Tipping Induced by Additive Control}

Now we add the assumption that the ODE $\dot{y} = f(y)$ has a hyperbolic attracting rest point at $y=a$.

\begin{proposition}
	Assuming $f$, $u$, $a$ as above with the additional condition that $\lim\limits_{t\to-\infty} u(t) = 0$, there exists a unique solution $\hat{y}(t)$ to the ODE $\dot{y} = f(y) + u(t)$ such that $\lim\limits_{t\to - \infty}\hat{y}(t) = a$. 
\label{prop:unique_soln}
\end{proposition}

\begin{proof}
	This claim is closely related to a fundamental result in R-tipping where the external force is assumed to be smooth and is not necessarily applied rigidly (see Theorem 2.2 in \cite{ashwinParameterShiftsNonautonomous2017}). In our context, an essentially identical argument carries through, and the co-moving coordinate change makes it slightly less burdensome. We state the proof in $n$ dimensions since it requires no additional effort over one dimension.
	
	First, assume without loss of generality that $a=0$; otherwise, translate the base vector field by replacing $y$ with $y-a$. Define \begin{align*}
		\omega(\epsilon) &= \sup \{|Df(y)-Df(0)|: |y|< \epsilon\}\\
		\delta(T) &= \sup\limits_{t<-T}|u(t)|
	\end{align*}
	Because $f\in C^2$ and $f(0) = 0$ we have $\omega(\epsilon)\to 0$ as $\epsilon \to 0$. Because $\lim\limits_{t\to-\infty} u(t) = 0$ we have $\delta(T)\to 0$ as $T \to \infty$.
	Hyperbolic stability of the equilibrium at $0$ means that there exist $K>0, \alpha>0$ such that
	\[|e^{At}| \leq Ke^{-\alpha t} \text{ for } t \leq 0\] where  $A = Df(0)$.
	Now define $h(y,t) = f(y) + u(t) - A y$ and rewrite the ODE as 
	\[\dot{y} = A y + h(y,t).\]
	
	Notice $D_yh = Df(y) - A = Df(y) - Df(0)$, so for all $t<-T$ we have \[\left|d_yh (y,t)\right| \leq \omega(|y|)\]
	\[\text{and } |h(0,t)| = |u(t)| \leq \delta(T) \text{ almost everywhere.}\]
	Now choose any $\Delta>0,  T_0> 0$ such that
	\[K\alpha^{-1}\omega(\Delta)\leq \frac{1}{2} \text{ and } K\alpha^{-1}\delta(T_0)\leq \frac{\Delta}{2}\]
	
	and consider the space of continuous functions
	\[S = \{y(t)\in C^0((\infty,-T_0]): |y(t)| \leq \Delta \text{ for } t < -T_0\}.\]
	
	We define an operator on $S$
	\[\Phi(y) = \int\limits_{-\infty}^t e^{A(t-s)}h(y(s), s) ~ds\]
	and verify first that it is well defined, second that it is a contraction mapping.
	
	For the former, 
	\begin{align*}|\Phi(y)(t)| &\leq \int\limits_{-\infty}^t Ke^{-\alpha(t-s)}[\delta(T_0) +|y(s)|\omega(|y(s)|)] ~ds\\
		&\leq \int\limits_{-\infty}^t Ke^{-\alpha(t-s)}[\delta(T) +\Delta\omega(\Delta)] ~ds\\
		&\leq K \alpha^{-1}\delta(T_0) + K\alpha^{-1}\Delta \omega(\Delta)\\
		&\leq \Delta
	\end{align*}
	
	and for the latter, since $h(y,t)$ is Lipschitz continuous in $y$ with Lipschitz constant $\omega(\Delta)$,
	\begin{align*}||\Phi(y_1)-\Phi(y_2)|| &= \sup\limits_{t\leq -T_0} \left | \int\limits_{-\infty}^t e^{A (t-s)}\left (h(y_1(s), s) -h(y_2(s), s)\right )~ds \right |\\
		&\leq \sup\limits_{t\leq -T_0}  \int\limits_{-\infty}^t Ke^{-\alpha (t-s)}\left |h(y_1(s), s) -h(y_2(s), s)\right |~ds\\
		&\leq K\alpha^{-1} \sup\limits_{t\leq -T_0}  \left |h(y_1(t), t) -h(y_2(t),t)\right |\\
		&\leq K\alpha^{-1}\omega(\Delta)||y_1 - y_2||\\
		&\leq \frac{1}{2}||y_1 - y_2||
	\end{align*}
	
	So $\Phi$ has a unique fixed point $\hat{y}$. By the variation of parameters formula, which still holds in our control setting, the fixed point $\hat{y}(t)$ is also the unique solution of the ODE that satisfies $|y(t)|\leq \Delta$ for all $t\leq T_0$. Since $\Delta$ can be chosen arbitrarily small, $\lim\limits_{t\to - \infty}\hat{y}(t) = 0$.
\end{proof}

Next, let $D\subset \mathbb{R}$ denote the basin of attraction of the attracting rest point at $y=a$, assume that its boundary $\partial D$ is nonempty. For a scalar system the basin is simply an interval, and the boundary consists either of one or two isolated points. We will assume the generically true property that the boundary points are hyperbolic.

Additionally, we let $u$ decay to 0 in forward time, and desire sufficiently fast decay such that solutions to the nonautonomous ODE limit nicely in forward time to solutions of the base autonomous ODE. For simplicity we assume that $u$ is eventually $C^1$ smooth with exponential decay; the reason is that with these conditions we may call upon a compactification technique developed in \cite{wieczorekCompactificationAsymptoticallyAutonomous2021} to obtain the forward limiting behavior. Though it should be possible to relax the smoothness assumption, we do not address this prospect here. 

\begin{proposition}
Let $\hat{y}(t)$ be the solution with $\lim\limits_{t\to - \infty}\hat{y}(t) = a$, whose existence and uniqueness were shown in Proposition \ref{prop:unique_soln}. Assume that the boundary $\partial D$ of the basin of attraction of the attractor at $y=a$ consists of either one or two hyperbolic unstable rest points. Assume that there exists a time $T$ such that $u$ restricted to $(T, \infty)$ is $C^1$ smooth. Also assume that $\lim\limits_{t\to\infty} u(t) = 0$ with exponential decay, meaning there exists a number $\rho$ such that $\lim\limits_{t\to\infty} \frac{\dot{u}(t)}{e^{-\rho t}}$ exists. Then $\hat{y}$ must exhibit exactly one of three long-term behaviors in forward time:

	\begin{itemize}
			\item $\lim\limits_{t\to \infty} \hat{y}(t) = a$
			\item $\lim\limits_{t\to \infty} \hat{y}(t) \in \partial D$
			\item $\hat{y}(t)$ escapes the closure $\overline{D}$ of the basin. That is, there exists a $T$ within the maximal interval of existence of the solution $\hat{y}(t)$ such that for all $t>T$ where $\hat{y}(t)$ is defined, $\hat{y}(t) \not\in \overline{D}$.	\end{itemize}
\label{prop:3_forward_time_behaviors}
\end{proposition}

\begin{proof}
	We leave this proof as a brief sketch, and direct the reader toward the sources \cites{wieczorekCompactificationAsymptoticallyAutonomous2021, wieczorekRateinducedTippingThresholds2023} for full details on the compactification procedure. In the future limiting autonomous system $\dot{y} =f(y)$, the listed behaviors comprise the only 3 possible behaviors for any solution. The $C^1$-smooth exponential decay of $u$ to zero allows the use of a compactification trick in forward time by "gluing on" the forward limiting autonomous system. This results in a smooth $(n+1)$-dimensional autonomous ODE where the hyperbolic basin boundary gains one unstable time dimension but remains hyperbolic. The exponential decay eliminates any pathological behaviors that might arise through compactifying. This construction provides a correspondence between forward behaviors in the glued-on cross section and forward behaviors of the original nonautonomous solutions. 
\end{proof}

\begin{definition}
	Assume $n=1$ and let $\hat{y}(t)$ be the unique solution with $\lim\limits_{t\to - \infty}\hat{y}(t) = a$. Out of the three possible forward behaviors from Proposition \ref{prop:3_forward_time_behaviors}, if it is not the case that $\lim\limits_{t\to  \infty}\hat{y}(t) = a$, then we say $u(t)$ \textbf{induces tipping} in Equation \eqref{eq:control}. If $\lim\limits_{t\to  \infty}\hat{y}(t) \in$ $\partial D$ we say that $u(t)$ is \textbf{critical}. 
	\label{def:tipping_critical_control}
\end{definition}

\subsection{Tipping Induced by Translational External Force}

Next, consider again the original rigidly shifting ODE 
\begin{equation}
	\dot{x} = f(x+\lambda(t)).
	\label{eq:rtipping}
\end{equation}
where we assume $\lambda(t):\mathbb{R}\to \mathbb{R}$ is globally Lipschitz continuous and satisfies the asymptotic conditions

\begin{itemize}
	\item $\lim\limits_{t\to-\infty} \lambda(t) = 0$.
	\item $\lim\limits_{t\to\infty} \lambda(t) = \lambda_{\infty} \text{ for a finite constant } \lambda_{\infty}$.
\end{itemize}

Lipschitz continuity of $\lambda$ implies absolute continuity of $\lambda$, which guarantees its almost-everywhere differentiability. The resulting measurable derivative $u=\dot{\lambda}$ is essentially bounded (actually, bounded) by the global Lipschitz constant of $\lambda$. 

The transformation
\begin{equation}
	y=x+\lambda
	\label{eq:coordinatechange}
\end{equation}
is absolutely continuous with absolutely continuous inverse, establishing a one-to-one, absolutely continuous correspondence between solutions of Equation \eqref{eq:rtipping} and solutions of
\begin{equation}
	\dot{y} = f(y) + \dot{\lambda}(t).
\end{equation}

Hence we obtain existence and uniqueness of local solutions to Equation \eqref{eq:rtipping} in the same extended sense as before. 

From the asymptotic conditions on $\lambda$ it follows that $\lim\limits_{t\to\pm\infty} u(t) = 0$. Add also now the assumption that there exists a $T$ such that $\lambda$ restricted to $(T, \infty)$ is $C^2$ smooth and the derivative $u = \dot{\lambda}$ decays exponentially as $t\to\infty$, so that the conditions for Propositions \ref{prop:unique_soln} and \ref{prop:3_forward_time_behaviors} are satisfied. Then it is straightforward to check the following proposition:

\begin{proposition}
	The solution $\hat{y}$ from Proposition \ref{prop:unique_soln} of Equation \eqref{eq:control} corresponds to a unique solution $\hat{x}$ of \eqref{eq:rtipping} such that $\lim\limits_{t\to  -\infty}\hat{x}(t) = a$, and the 3 cases from Proposition \ref{prop:3_forward_time_behaviors} correspond respectively to 3 long term behaviors for $\hat{x}$ in forward time:
	\begin{itemize}
		\item $\lim\limits_{t\to \infty}\hat{x}(t) = a-\lambda_{\infty}$
		\item $ \lim\limits_{t\to \infty}\hat{x}(t) \in \partial D-\lambda_{\infty}$
		\item $\hat{x}(t) \text{ escapes } \overline{D}-\lambda_{\infty}.$
	\end{itemize}
where $S-\lambda_{\infty}$ denotes the set $\{s - \lambda_{\infty} ~|~ s\in S\subset \mathbb{R}^n\}$
\label{prop:unique_soln_three_forward_behaviors_untransformed_x}
\end{proposition}

\begin{definition}
	We say $\lambda$ \textbf{induces tipping} in Equation \eqref{eq:rtipping} if and only if $u= \dot{\lambda}$ induces tipping in Equation \eqref{eq:control}. And similarly we say $\lambda$ is \textbf{critical} if and only if $u$ is critical.
\end{definition}

\begin{remark}
	For an $r$-parameterized family of smooth scalar ramp functions $\lambda$, rate-induced tipping is typically defined via end-point tracking/non-tracking of a quasi-static equilibrium \cite{ashwinParameterShiftsNonautonomous2017}, or as a nonautonomous bifurcation of a pullback attractor that loses its forward attraction \cite{hoyer-leitzelRethinkingDefinitionRateinduced2021}. In our present application, we require no more than the simply stated definition above. It is equivalent to the quasi-static equilibrium and pullback attractor definitions found in the literature for the parameterized smooth case. 
\end{remark}

\section{Monotone Ramp, One-Sided Basin Version}

Let us initially restrict our attention to the case where $D$ is half infinite. Without loss of generality, assume $$D = (-\infty, \beta) \text{ and } \beta < \infty.$$ 

Additionally, let us initially assume that $\lambda(t)$ is monotone non-decreasing, hence $$u(t) = \dot{\lambda}(t) \geq 0 \text{ wherever it is defined.}$$ Both these restrictions, that $D$ be half infinite and that $\lambda$ be monotone, are primarily for the sake of ease in the initial exposition. Afterward we explain how remove both these assumptions, which requires slight adjustment to the statement of the result. 

\begin{figure}[H]
	\begin{center}
		
		\begin{tikzpicture}
			\begin{axis}[axis lines=center,
				xlabel=$y$,
				xmin=-15,
				xmax=15,
				ymin=-10,
				ymax=5,
				xscale=1.2,
				yscale=0.7,
				xtick={-8.22726, 10.5997},
				ytick={-8.47549},
				hide obscured x ticks=false,
				hide obscured y ticks=false,
				major tick length=0,
				xticklabels={$a$, $\beta$},
				yticklabels={$-\mu$},
				xticklabel style={font=\small, xshift=-1.5ex, color=black},
				every tick/.style={black, thick},
				]
				\addplot[black, mark=*,mark options={yscale=1.2, xscale=0.7}] coordinates{
					(-8.22726, 0)
					(10.5997,0)};
				\addplot[smooth, thick, black, domain=-12:12] {0.1*(((-0.3*x) + 3)*(-0.3*x-2)^2*(-0.3*x + 1)^3 - 50)};
				\addplot[smooth, thick, black, dotted, domain=0:8.57795] {-8.47549};
				\node[] at (9,3.5) {$f(y)$};
			\end{axis}
		\end{tikzpicture}
		
	\end{center}
	\caption{An arbitrary scalar vector field $f(y) \in C^2$ with an attracting rest point at $y=a$ whose basin of attraction is the half open interval $(-\infty, \beta)$. The minimum value of $f(y)$ on $[a, \beta]$ is denoted $-\mu$.}
\end{figure}
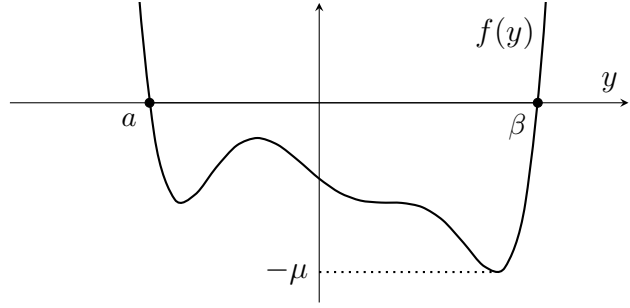

Let $-\mu$ be the minimum value of $f$ on $[a,\beta]$. For any constant $M > \mu$ consider the initial value problem 
\begin{gather*}
	\dot{y} = f(y)+ M\\
	y(0)=a
\end{gather*}

Since the right hand side of the ODE is positive for all $y\in [a, \beta]$, the solution $y(t)$ of the initial value problem is strictly increasing there. Clearly a unique time $T_M>0$ exists such that $y(T_M) = \beta$. Define the following bang-bang control function
	\begin{equation}B_M(t) = \begin{cases}
	0 & \text{if } t < 0\\
	M & \text{if } 0 \leq t \leq T_M\\
	0 & \text{if } T_M < t
	\label{eq:bang_bang}
\end{cases}.\end{equation}

By construction $B_M(t)$ is critical in the sense of Definition \ref{def:tipping_critical_control}. That is, it steers the initial condition $y=a$ to an end state exactly balanced on the boundary $y=\beta$ of the basin of attraction. 

We now introduce the optimization problem for which we will claim that $B_M(t)$ is an optimal solution. 

\begin{problem} Fix a constant $M > \mu$ as before. Call a pair $(y(t), u(t))$ an \textbf{admissible pair} if $y$ is absolutely continuous on $\mathbb{R}$, $u$ is measurable on $\mathbb{R}$, and they solve the ODE $\dot{y} = f(y)+u(t)$ subject to the constraints:
\begin{itemize}
	\item $\lim\limits_{t\to -\infty}y(t) = a$,
	\item $\lim\limits_{t \to \infty}y(t) = \beta$,
	\item $u(t) \in [0,M]$ for almost every $t$.
\end{itemize}
In this case, we also say $u$ is an \textbf{admissible control} and $y$ is the corresponding \textbf{admissible trajectory}. For an admissible pair $(y,u)$, if it achieves a global minimum value of the integral $$\int_{-\infty}^{\infty} u(t) dt$$ among all admissible pairs, then $(y,u)$ is called an \textbf{optimal pair}. 
In this case, we also say $u$ is an \textbf{optimal control} and $y$ is the corresponding \textbf{optimal trajectory}.
\label{prob:optimal_control_scalar_half_basin_monotone}
\end{problem}

\begin{lemma} The bang-bang function $B_M$ defined in \eqref{eq:bang_bang} is an optimal control for Problem \ref{prob:optimal_control_scalar_half_basin_monotone}, and every optimal control is equal to $B_M$ almost everywhere and up to time translation. 
	\label{lemma:scalar_optimal}
\end{lemma}

\begin{proof}
	Note if an optimal control exists, the corresponding optimal state trajectory $y(t)$ must be strictly increasing during all times $t$ such that $y(t) \in (a,\beta)$. Otherwise, there would exist $t_1 < t_2$ with $y(t_1) = y(t_2)$, $y(t) \in (a, \beta)$ for all $t \in  [t_1, t_2]$, and $\dot{y}(t) \geq 0$ on a subset of positive measure of $[t_1, t_2]$. Because $f(y(t)) < 0$ for all $t \in  [t_1, t_2]$ we must have $u(t) > 0$ on that same subset of positive measure of $[t_1, t_2]$. By definition $u(t) \geq 0$ everywhere. So by entirely excising the interval $[t_1, t_2)$ we produce a strictly lower cost admissible control. 

Restricting to $t$ such that  $y(t)\in (a,\beta)$, we have that $y(t)$ is invertible and $\dot{y}(t)>0$ for almost all $t$. We can now directly compute a lower bound on the integral \begin{align*}
		& \int\limits_{-\infty}^{\infty} u(t) ~ dt \\
		= & \int\limits_{-\infty}^{\infty} \dot{y}(t)-f(y(t)) ~ dt \\
		= & \int\limits_{-\infty}^{\infty} \dot{y}(t) ~ dt-\int\limits_{-\infty}^{\infty}f(y(t)) ~ dt \\
		=& (\beta-a) -\int\limits_{-\infty}^{\infty}f(y(t)) ~ dt\\
		=& (\beta-a) -\int\limits_{a}^{\beta}\frac{f(y)}{\dot{y}} ~ dy \\ 
		=& (\beta-a) -\int\limits_{a}^{\beta}\frac{f(y)}{f(y) + u(t(y))} ~ dy \\
		\geq & (\beta-a) -\int\limits_{a}^{\beta}\frac{f(y)}{f(y) + M} ~ dy
	\end{align*}
	The latter inequality follows from the facts that $\dot{y} = f(y) +u >0$ and $f(y)<0$, thus \begin{align*}
	& 0 < f(y) +u \leq f(y)+M \\
	\implies& \dfrac{1}{f(y) +u} \geq \dfrac{1}{f(y) +M} \\
	\implies& \dfrac{-f(y)}{f(y) +u} \geq \dfrac{-f(y)}{f(y) +M}
\end{align*}
	
This lower bound on the value of the cost integral is achieved exactly by the bang-bang control $B_M$, hence we conclude $B_M$ is an optimal control. Any other optimal control must be equal to $M$ at almost all times when $y \in (a,b)$. Clearly the optimal strategy outside of this is to set $u = 0$ almost everywhere. Thus any optimal control is equal to $B_M$ almost everywhere and up to time-translation.
	\end{proof}
	
	\begin{lemma}
		The integral of the bang-bang control function \eqref{eq:bang_bang}, which is $$\int\limits_{-\infty}^{\infty}B_M(t) dt = M \cdot T_M,$$ is a strictly decreasing continuous function of $M$. Additionally, its limiting behavior satisfies
		\begin{itemize}
			\item $\lim\limits_{M\to \infty} M  \cdot T_M = \beta - a$,
			\item $\lim\limits_{M\to \mu^+} M  \cdot T_M = \infty$.
		\end{itemize}
		\label{lemma:decreasing}
	\end{lemma}
	
	\begin{proof}
	Take two different values $\mu < M_1 < M_2$ and compare the respectively associated bang-bang functions $B_{M_1}$ and $B_{M_2}$. Set $M=M_2$ in the optimization problem (Problem \ref{prob:optimal_control_scalar_half_basin_monotone}), so that by Lemma \ref{lemma:scalar_optimal} $B_{M_2}$ is an optimal control and $B_{M_1}$ is a strictly suboptimal control. This yields the strict decreasing order $\int_{-\infty}^{\infty} B_{M_1}(t) dt > \int_{-\infty}^{\infty} B_{M_2}(t) dt$.
	
	Continuity of $M\cdot T_M$ follows if the switching time $T_M$ when $y$ arrives at the terminal state $\beta$ is continuous in $M$. This follows from a well known property of globally continuous dependence on parameters for solutions to initial value problems with a globally Lipschitz vector field. Here, $\dot{y} = f(y)+M$ is locally Lipschitz on $\mathbb{R}$ thus globally Lipschitz on the compact set of interest $y \in [a, \beta]$.

	For the first limit, recall $T_M$ is defined so that $\beta - a = \int_{0}^{T_M} f(y(s)) + M ~ds$. Since $-\mu \leq f(y(s)) \leq 0$ we have 
	\begin{align*}
	& T_M(M-\mu) \leq \beta - a \leq  MT_M\\
		\implies& -\mu T_M \leq (\beta - a) - MT_M \leq  0
	\end{align*}
	and taking the limit as $M\to\infty$ gives $T_M\to0$ and $(\beta - a) - MT_M \to 0$.

	For the second limit, it suffices to show that $\lim\limits_{M\to \mu^+} T_M = \infty$. Let $$y_\mu = \min\{y \in [a, \beta] ~:~ f(y) = -\mu\}$$ be the first point at which $f(y)$ achieves its minimum value in $[a, \beta]$. Then $y_\mu$ is a rest point of the ODE $\dot{y} = f(y) + \mu$ such that the solution to the initial value problem $y(0) = a$ approaches $y_\mu$ as $t\to\infty$. By globally continuous dependence of solutions on parameters, we may choose $M \approx \mu$ such that the solution of $\dot{y} = f(y) + M, ~y(0) = a$ takes an arbitrarily large time to reach $y_\mu$. 
\end{proof}

The final ingredient is the next lemma, which is modeled on the piecewise linear ramping example (Example \ref{ex:piecewise}) from the beginning of this chapter, except that it replaces the example base vector field with our arbitrary one $\dot{x} = f(x)$. Exactly the same tipping behavior still occurs, though of course the value of the critical slope $m_c$ depends on the choice of $f$. 

\begin{lemma}
Fixing an $f$ as before with attracting rest point at $a$ and basin boundary $\beta$, and a constant $\lambda_{\infty} > \beta-a$, consider the following parameterized family of piecewise linear ramp functions with parameter $m>0$:
\[
\lambda(mt) = \begin{cases}
0 & \text{if } t < 0\\
mt & \text{if } 0 \leq t \leq \lambda_{\infty}/m\\
0 & \text{if } t > \lambda_{\infty}/m
\end{cases}.
\]

There exists a unique solution $\hat{x}(t)$ of the ODE $\dot{x} = f(x+\lambda(mt))$ such that $\lim \limits_{t\to -\infty} \hat{x}(t) = a$, and there exists a unique critical parameter value $m=m_c$ such that

\[\begin{cases}
\lim\limits_{t\to \infty}\hat{x}(t) = a-\lambda_{\infty} & \text{for } m < m_c\\
\lim\limits_{t\to \infty}\hat{x}(t) = \beta-\lambda_{\infty} & \text{for } m=m_c\\
\hat{x}(t) \text{ escapes } \overline{D}-\lambda_{\infty} & \text{for } m > m_c
\end{cases}.
\]

Furthermore, when considered as a function of $\lambda_{\infty}$, $m_c$ is continuous and strictly decreasing and $\lim_{\lambda_{\infty}\to\beta-a} m_c = \infty$, $\lim_{\lambda_{\infty}\to\infty} m_c = \mu$ where $-\mu$ is the minimum value of $f$ on $[a, \beta]$.
\label{lemma:critical}
\end{lemma}

\begin{proof}
	
Proposition \ref{prop:unique_soln_three_forward_behaviors_untransformed_x} established the existence of the unique solution $\hat{x}(t)$ and fact that the three forward behaviors mentioned are the only possible forward behaviors for $\hat{x}(t)$. Under the co-moving change of coordinates, the ramp $\lambda(mt)$ is transformed into a control function $u(mt) = \dot{\lambda}(mt)$ which is a bang-bang style step function where the step has height $m$ and width $\lambda_\infty/m$. First of all, if $m \leq \mu$ then $u$ certainly cannot induce tipping, thus the first option out of the three forward behavior occurs when $m \leq \mu$. Now assuming $m > \mu$, compare $u(mt)$ to the critical bang-bang function $B_m(t)$ \eqref{eq:bang_bang}, whose step also has height $m$ but has a possibly different width $T_m$. By definition of $T_m$, we see that $u(mt)$ induces tipping if and only if $T_m \leq \lambda_{\infty}/m$, with equality giving criticality. Rewriting slightly, $u(mt)$ induces tipping if and only if $m T_m \leq  \lambda_{\infty}$. 

By Lemma 	\ref{lemma:decreasing}, $mT_m$ is a continuous decreasing function of $m$ with range $(\beta -a, \infty)$; thus it intersects the constant $\lambda_{\infty} >\beta-a$ exactly once, which gives $m=m_c$ with the desired behaviors on either side of $m_c$, as well as the strict decreasing property, the continuity, and the limiting behaviors.
\begin{figure}[H]
	\begin{center}
		\begin{tikzpicture}
			\begin{axis}[axis lines=center,
				xlabel=$m$,
				xmin=-1,
				xmax=10,
				ymin=-1,
				ymax=10,
				xtick={1.5, 3},
				ytick={1, 2.3333},
				yscale=0.8,
				hide obscured x ticks=false,
				hide obscured y ticks=false,
				major tick length=5,
				xticklabels={$\mu$, $m_c$},
				yticklabels={$\beta -a$, $\lambda_{\infty}$},
				xticklabel style={font=\small, xshift=-1.5ex, color=black},
				every tick/.style={black, thick},
				]
  				\addplot[black, mark=*,,mark options={yscale=1.2}] coordinates{(3, 2.3333)};
				\addplot[smooth, thick, black, samples=100, domain=1.5:10] {1+2/(x-1.5))};
				\addplot[smooth, thick, black, dotted, domain=0:10] {1};
				\addplot[smooth, thick, black, domain=0:10] {2.3333};
   				\addplot[dotted, thick] coordinates {(1.5, 0) (1.5, 10)};
   				\addplot[dotted, thick] coordinates {(3, 0) (3, 2.333)};
				\node[] at (3,9.5) {$m\cdot T_m$};
			\end{axis}
		\end{tikzpicture}
		
	\end{center}
	\caption{The continuous curve $m T_m$ is strictly decreasing in $m$ and approaches $\infty, \beta-a$ in the limits as $m \to \mu, \infty$, respectively. The curve crosses the constant $\lambda_\infty > \beta -a$ exactly once, giving the desired critical slope $m_c$. As $\lambda_{\infty}\to\beta-a$ we have $m_c \to \infty$; as $\lambda_{\infty}\to\infty$ we have $m_c \to \mu$.}
\end{figure}
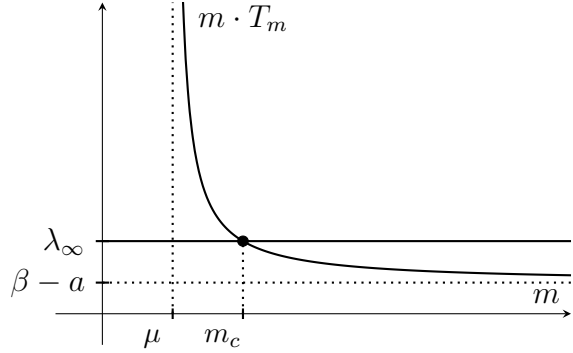
\end{proof}

\begin{theorem}[Scalar, Monotone Ramp, One-Sided Basin Version]
	Assume that $f: \mathbb{R} \to \mathbb{R}$ is $C^2$ and the ODE $\dot{x} = f(x)$ has an attracting rest point at $x=a$ whose basin of attraction $D$ is the half-infinite interval $D = (-\infty, \beta)$ where $\beta > a$ is a finite number, and $x=\beta$ is a hyperbolic unstable rest point of $\dot{x} = f(x)$. Fix a finite constant $\lambda_{\infty}> \beta-a$, and assume $\lambda:\mathbb{R}\to \mathbb{R}$ is globally Lipschitz continuous and monotone non-decreasing with $\lim\limits_{t\to-\infty} \lambda(t) = 0$, $\lim\limits_{t\to\infty} \lambda(t) = \lambda_{\infty}$. Assume there exists a $T$ such that $\lambda$ is $C^2$ when restricted to $(T, \infty)$, and there exists a number $\rho$ such that $\lim\limits_{t\to\infty} \frac{\ddot{\lambda}(t)}{e^{-\rho t}}$ exists. Let $-\mu<0$ equal the minimum value of $f$ on $[a, \beta]$. Then there exists a number $m_c>\mu$ such that if $\lambda(t)$ induces tipping in the ODE $\dot{x} = f(x+\lambda(t))$ then $\dot{\lambda}(t) \geq m_c$ at least once. Moreover, there exists a choice of $\lambda(t)$ satisfying the given conditions which does induce tipping with $\max_t \dot{\lambda}(t) = m_c$. Finally, $m_c$ is continuous and strictly decreasing when viewed as a function of $\lambda_\infty$ and satisfies $\lim_{\lambda_{\infty}\to\beta-a} m_c = \infty$, $\lim_{\lambda_{\infty}\to\infty} m_c = \mu$.
	\label{thm:special_scalar}
\end{theorem}

\begin{proof}
Choose $m_c$ as defined in Lemma \ref{lemma:critical}. The critical piecewise linear ramp function that it corresponds to in Lemma \ref{lemma:critical} induces tipping with maximum slope $m_c$. To show that $\dot{\lambda}(t) \geq m_c$ at least once, a slightly different argument is used depending on whether $\dot{\lambda}(t)$ attains its supremum or not. 

\begin{enumerate}
	\item[Case 1.]  $\dot{\lambda}(t)$ attains its supremum. 
	
	Suppose for contradiction that $\lambda(t)$ induces tipping but $\max_t \dot{\lambda}(t) = N < m_c$. Let $u = \dot{\lambda}$ and assume $u$ is critical -- otherwise, truncate it (and set equal to 0) on the right end while decreasing its integral; thus, $\int_{-\infty}^{\infty} u(t) ~dt \leq \lambda_{\infty}$. 

Note $u$ is now an admissible control for the optimization problem (Problem \ref{prob:optimal_control_scalar_half_basin_monotone}) when the control constraint is $u \in [0,N]$. So by optimality of the bang-bang function $B_N$ in Lemma \ref{lemma:scalar_optimal}, its integral satisfies
\begin{align*}
\int\limits_{-\infty}^{\infty} B_N(t) ~dt &\leq \int\limits_{-\infty}^{\infty} u(t) ~dt\\
&  \leq \lambda_{\infty}
\end{align*}
But also \begin{align*}
	\int\limits_{-\infty}^{\infty} B_N(t) ~dt &> \int\limits_{-\infty}^{\infty} B_{m_c}(t) ~dt \text{ by Lemma \ref{lemma:decreasing}}\\
	&  = \lambda_{\infty} \text{ by definition}
\end{align*}
so we have reached a contradiction.

\item[Case 2.] $\dot{\lambda}(t)$ does not attain its supremum. 

Suppose for contradiction that $\lambda(t)$ induces tipping but $\sup_t \dot{\lambda}(t) \leq m_c$. Let $u = \dot{\lambda}$ and assume $u$ is critical -- otherwise, truncate it (and set equal to 0) on the right end while decreasing its integral; thus, $\int_{-\infty}^{\infty} u(t) ~dt \leq \lambda_{\infty}$. 

Note $u$ is now an admissible, but strictly suboptimal, control for the optimization problem (Problem \ref{prob:optimal_control_scalar_half_basin_monotone}) when the control constraint is $u \in [0,m_c]$. So by optimality of the bang-bang function $B_{m_c}$ from Lemma \ref{lemma:scalar_optimal}, its integral satisfies
\begin{align*}
	\int\limits_{-\infty}^{\infty} B_{m_c}(t) ~dt &< \int\limits_{-\infty}^{\infty} u(t) ~dt\\
	&  \leq \lambda_{\infty}
\end{align*}
But also \begin{align*}
	\int\limits_{-\infty}^{\infty} B_{m_c}(t) ~dt  = \lambda_{\infty} \text{ by definition}
\end{align*}
so we have reached a contradiction.
\end{enumerate}
\end{proof}

\section{General Scalar Version}

At this point, we discuss how to expand from the above special case to a general scalar version where assumptions of the monotonicity of the ramp and the one-sidedness of the basin are removed. In summary,
\begin{itemize}
	\item Allowing a two-sided boundary for the basin of attraction. An optimal escape trajectory would traverse through only one side of the basin; it would not cross back over the attractor. 
	\item Removing monotonicity. This requires a notable revision to our interpretation of the amplitude of the perturbation. By amplitude we no longer mean the forward limiting constant $\lambda_\infty$ but instead $$L = \int\limits_{-\infty}^{\infty} |\dot{\lambda}(t)| ~dt,$$ the total \textbf{arclength} of the perturbation, which is at least as large as $\lambda_\infty$, with equality if $\lambda$ is monotone. 
\end{itemize}

\begin{remark}
	Since a non-monotone $\lambda$ may no longer bear a resemblance to a ramp, we refer to it as an \textbf{external forcing function} rather than a ramp function.
\end{remark}

\begin{theorem}[General Scalar Version]
	Assume that $f: \mathbb{R} \to \mathbb{R}$ is $C^2$ and the ODE $\dot{x} = f(x)$ has an attracting rest point at $x=a$ whose basin of attraction $D$ has a boundary consisting of either one or two hyperbolic unstable rest points. Write $D=(\alpha, \beta)$, where $-\infty \leq \alpha < a < \beta \leq \infty$ and at least one of $\alpha, \beta$ is finite. Let $R$ be the radius of the basin, that is $R = \min\{a-\alpha,\beta-a\}$. Define
		\begin{align*}
			\mu_-&=\begin{cases}
			\max\{f(x) ~|~ x \in [\alpha, a]\} \text{ if } \alpha \neq -
			\infty\\
			\infty \text{ otherwise} 
		\end{cases}\\
		\mu_+&=\begin{cases}
			-\min\{f(x) ~|~ x \in [a,\beta]\} \text{ if } \beta \neq 
			\infty\\
			\infty \text{ otherwise} 
		\end{cases}
	\end{align*} and let $\mu = \min\{\mu_-, \mu_+\}$. Thus $\mu$ is the minimum of the maximums of the magnitude $|f|$ of the vector field over each escapable side of the basin. Fix a constant $L> R$, and assume the external forcing function $\lambda:\mathbb{R}\to \mathbb{R}$ is globally Lipschitz continuous with $\lim\limits_{t\to-\infty} \lambda(t) = 0$, $\lim_{t\to\infty}\lambda(t)$ finite, and $\int\limits_{-\infty}^{\infty} |\dot{\lambda}(t)| ~dt = L$. Assume there exists a $T$ such that $\lambda$ is $C^2$ when restricted to $(T, \infty)$, and there exists a number $\rho$ such that $\lim\limits_{t\to\infty} \frac{\ddot{\lambda}(t)}{e^{-\rho t}}$ exists. Then there exists a number $m_c>\mu$ such that if $\lambda(t)$ induces tipping in the ODE $\dot{x} = f(x+\lambda(t))$ then $|\dot{\lambda}(t)| \geq m_c$ at least once.  Moreover, there exists a choice of $\lambda(t)$ satisfying the given conditions which does induce tipping with $\max_t \dot{\lambda}(t) = m_c$. Finally, $m_c$ is continuous and strictly decreasing when viewed as a function of $L$ and satisfies $\lim_{L\to R} m_c = \infty$, $\lim_{L\to\infty} m_c = \mu$.
	\label{thm:gen_scalar}
\end{theorem}

\begin{proof}
We sketch the proof. First we require an updated optimization problem where we add magnitude bars around the control, and also update the second constraint in order to allow admissible trajectories to forward limit to either of the boundary points, if there are two:

Minimize the cost functional $$\int_{-\infty}^{\infty} |u(t)| dt$$ subject to the constraints 
\begin{itemize}
	\item $\lim\limits_{t\to -\infty}y(t) = a$,
	\item $\lim\limits_{t \to \infty}y(t) \in \partial D$,
	\item $|u(t)| \in [0,M]$ for almost every $t$.
\end{itemize}
There are up to two possible escape paths, and one can show the optimal control is a positive or negative bang-bang function applied to achieve escape over whichever side of the basin is "easier." Thus the optimal value $J(M)$ of the cost functional is a pointwise minimum of up to two continuous decreasing functions, which is itself continuous and decreasing. 

There are limits for the optimal cost that arise the same way as in Lemma \ref{lemma:decreasing} over each of the up to two paths. In particular, as $M\to\infty$ the optimal cost on each relevant path limits to the length of the path, so overall $\lim\limits_{M \to \infty} J(M)=R$, the minimum path length. On the other hand, the optimal cost on each relevant path limits to infinity as $M$ approaches the maximum magnitude $\mu_P$ of $|f|$ on that path, so $\lim\limits_{M \to \min \mu_P} J(M) = \lim\limits_{M \to \mu} J(M) = \infty$. 

Now every $L>\beta-R$ intersects the curve $J(M)$ exactly once, giving $m_c$. Notice that if a threshold speed were defined similarly to that in Lemma \ref{lemma:decreasing}, but possibly one on each side of the basin, then $m_c$ here would be their minimum. The rest of the argument proceeds with exactly the same logic as in the proof of Theorem \ref{thm:special_scalar}, replacing $\lambda_\infty$ with $L$,  replacing $B_N$ and $B_{m_c}$ with the new optimal controls for these problems respectively, and inserting magnitude bars where appropriate. 
\end{proof}

	\newpage
	\bibliography{refs}

\end{document}